\RequirePackage{ifthen}
\newboolean{SIOPT}
\setboolean{SIOPT}{false}

\ifthenelse {\boolean{SIOPT}}
{
\documentclass[review]{siamart0516}
} {
\documentclass[10pt]{article}
\usepackage[margin=1.35in]{geometry}
}

\usepackage{graphicx}
\usepackage{amsmath}
\usepackage{amssymb}
\usepackage{hyperref}
\usepackage{color}
\usepackage{multirow}

\usepackage{mathrsfs} 

\usepackage{soul} 

\usepackage{stmaryrd} 

\usepackage[makeroom]{cancel} 

\newcommand{\R}{\mathbb R}
\newcommand{\Z}{\mathbb Z}

\newcommand{\N}{\mathbb N}

\DeclareMathOperator{\proj}{proj}
\DeclareMathOperator{\supp}{supp}
\DeclareMathOperator{\opt}{Opt}
\DeclareMathOperator{\val}{Val}
\DeclareMathOperator{\aug}{Aug}
\DeclareMathOperator{\ext}{ext}
\DeclareMathOperator{\dist}{dist}
\DeclareMathOperator{\suc}{s.t.}
\DeclareMathOperator{\cov}{cov}

\newcommand{\D}{\mathcal D}
\renewcommand{\L}{\mathcal L}
\renewcommand{\S}{\mathcal S}
\newcommand{\X}{\mathcal X}

\newcommand{\norm}[1]{\left\lVert#1\right\rVert}
\newcommand{\pare}[1]{\left(#1\right)}
\newcommand{\abs}[1]{\left\lvert#1\right\rvert}
\newcommand{\es}{\sigma}

\newcommand{\notemi}[1]{}

\usepackage{comment}
\excludecomment{knownproof}

\ifthenelse {\boolean{SIOPT}}
{

\newenvironment{cpf}
{\begin{trivlist} \item[] {\hspace{.5cm} \em Proof of claim. }}
{$\hfill\diamond$ \end{trivlist} \smallskip}

\newcounter{step}
\newenvironment{step}[1][]
{\refstepcounter{step} \smallskip \bfseries Step~\thestep. #1}
{. }

} {
\usepackage{amsthm}
\newtheorem{theorem}{Theorem}
\newtheorem{proposition}{Proposition}
\newtheorem{lemma}{Lemma}
\newtheorem{definition}{Definition}

\newenvironment{cpf}
{\begin{trivlist} \item[] {\em Proof of claim. }}
{$\hfill\diamond$ \end{trivlist} \smallskip}

\newcounter{step}

}

\newtheorem{claim}{Claim}

\newcommand{\kywrds}{subset selection, 
linear least squares, polynomial-time algorithm, sparsity}

\begin{document}

\title{Subset Selection in Sparse Matrices\thanks{\textbf{Funding: }A.~Del~Pia is partially funded by  ONR grant N00014-19-1-2322. The work of R.~Weismantel is supported by the Einstein Foundation Berlin. Any opinions, findings, and conclusions or recommendations expressed in this material are those of the authors and do not necessarily reflect the views of the Office of Naval Research.}}

\author{Alberto Del Pia\thanks{Department of Industrial and Systems Engineering \& Wisconsin Institute for Discovery,
University of Wisconsin-Madison. 
E-mail: {\tt delpia@wisc.edu}.}
\and
Santanu S. Dey\thanks{School of Industrial and Systems Engineering, Georgia Institute of Technology.
E-mail: {\tt santanu.dey@isye.gatech.edu}.}
\and
Robert Weismantel\thanks{Department of Mathematics, ETH Z\"urich.
E-mail: {\tt robert.weismantel@ifor.math.ethz.ch}.}
}

\date{\today}

\maketitle







\begin{abstract}
In Subset Selection we search for the best linear predictor that involves a small subset of variables. 
From a computational complexity viewpoint, Subset Selection is NP-hard and few classes are known to be solvable in polynomial time.
Using mainly tools from Discrete Geometry, we show that some sparsity conditions on the original data matrix allow us to solve the problem in polynomial time.
\end{abstract}

\ifthenelse {\boolean{SIOPT}}
{

}{
}

\section{Introduction}

With the current developments in technology,  data science  and machine learning it is apparent that in the near future more and more variants of Mixed Integer Optimization models will emerge whose dimension is extremely large, if not huge.  In view of the fact that mixed integer optimization is NP-hard, the aim is to identify parameters in the input data that lead to polynomial-time algorithms. 


One natural parameter  in this context is the sparsity of the input data. More precisely,  the goal is to identify
 sparsity patterns in the  data and exploit them to design algorithms that depend more on the complexity of the sparsity patterns and less on the ambient dimension of the
	problem. The hope is that in such a way, high dimensional problems, which are
	typically intractable simply due to their dimension, can be made algorithmically
	solvable by decomposing them into smaller ones. 
For example, consider a Linear Integer Optimization
	problem where one aims at  optimizing a linear function over the integer points in a polyhedral region $\{x \in \Z^n \mid Ax \leq b\}$. Assume that the   underlying constraint matrix $A$ is the so-called $1$-sum of matrices $A^1,\ldots,A^h$, i.e., 	
\begin{equation*}
A= 
\left(\begin{array}{cccc}
A^1 & & \multicolumn{2}{c}{\multirow{2}{*}{0}}\\
& A^2 & & \\
\multicolumn{2}{c}{\multirow{2}{*}{0}} & \ddots & \\
& & & A^h \end{array}\right).
\end{equation*}
A high dimensional Linear Integer Optimization problem  of $1$-sum structure can be easily solved by decomposing it  into optimization problems associated with the smaller units $A^1,\ldots,A^h$.  In particular, if for every $i \in \{1,\ldots,h\}$ the number of columns of $A^i$ is a constant, then Lenstra's algorithm \cite{Lenstra1983}  applied $h$ times  shows us how to solve the original high-dimensional optimization problem in polynomial time.   
We remark that in this paper we use the standard notion of polynomial-time algorithms in Discrete Optimization, and we refer the reader to 
the book~\cite{SchBookIP} for a thorough introduction.
A more involved sparsity structure in a matrix arises when there is an additional matrix $C=(c_1|\ldots | c_k)$ whose number of columns $k$ is also assumed to be constant, but the matrix $C$ itself  is part of a larger matrix $M$ such that it may interact with all the matrices $A^i$. More formally,  the matrix $M$ is of the form
\begin{align}
\label{eq: matrix M}
M = (A|C) = 
\left(
\begin{array}{cccc|ccc}
A^1&&&&|&&| \\
&\ddots &&& c_1 & \cdots & c_k \\
&&A^h&&|&&|\\
\end{array}
\right).
\end{align}	
The complexity for Linear Integer Optimization associated with a matrix of sparsity pattern~\eqref{eq: matrix M}  in which the number $n_i$ of columns of each $A^i$ is a constant and $k$ is a constant is not known, unless one also requires that $\norm{M}_\infty$  is  a constant. 
In fact, extensive research in the field in the past decade on N-fold matrices~\cite{DLHOW2008, HOR2013}, $4$-block decomposable matrices~\cite{HKW2010} and tree-fold matrices~\cite{Chenetal} finally led to the most general result about which matrices of the form \eqref{eq: matrix M} can be handled efficiently: Linear Integer Optimization can be solved in polynomial time if the so-called tree width of $M$ or $M^T$ is constant and
the value $\norm{M}_\infty$ is a constant, see \cite{Kourtetal}.  
What can one say about other variants of Mixed Integer Optimization associated with matrices of sparsity pattern \eqref{eq: matrix M}  in which the number of columns of each $A^i$ is a constant and $k$ is a constant, but the parameter $h$ is part of the input?

One important class of Mixed Integer Optimization problems arising in the context of statistics, data science and machine learning, is called Subset Selection and is formally introduced in Equation \eqref{pr: main} below.   In  the underlying model  all variables may attain real values. Among all variables,
however,  only a certain, a-priori given number of them are allowed to attain
non-zero values. This latter combinatorial condition is the discrete flavor of the underlying model. The objective is a convex quadratic function in variables $x \in \R^d$ and the scalar variable $\mu$.
The remaining characters stand for rational data in the problem instance: $M$ is an $m \times d$ matrix, $b$ and $c$ are $m$-vectors, and $\es$ is a natural number.
Finally, the support of a vector $x$ is defined to be  $\supp(x) = \{i \mid x_i \neq 0\}$ and  $\norm{\cdot}$ denotes the Euclidean norm.
Subset Selection is then defined as the optimization problem
\begin{align}
\label{pr: main}
\begin{split}
\min \quad & \norm{Mx + c \mu -b} \\
\text{s.t.} \quad & x \in \R^d, \ \mu \in \R \\
& |\supp(x)| \le \es.
\end{split}
\end{align}

Problem~\eqref{pr: main} is a standard model in statistics, see the textbooks \cite{miller2002subset} and \cite{HasTibFri}. In this context, it is usually assumed that $c$ is the vector of all ones.
It is also often assumed that the columns of $M$ and $b$ are mean-centered (i.e., the sum of entries in the columns of $M$ and in $b$ is zero), in which case it can be shown that the optimal value of $\mu$ is zero.
From an optimization point of view there is no complication if one does not assume mean-centering. Therefore, we retain the $\mu$ variable and the vector $c$.

From a computational complexity point of view, Subset Selection  is -- such as Integer Linear Optimization -- NP-hard in general \cite{Wel82}.
This complexity viewpoint then leads us naturally to the same question introduced in the context of Linear Integer Optimization: can we identify sparsity patterns in the input data for Subset Selection that allow us to design algorithms that depend more on the complexity of the sparsity patterns and less on the ambient dimension of the problem? 
More precisely, given a  matrix of sparsity pattern \eqref{eq: matrix M} where each of the blocks $A^1,\ldots,A^h$ and $C$ have a constant number of columns, can Subset Selection be solved in polynomial time?
This paper answers the question in the affirmative. While some other sparsity patterns have been explored in the context of compressive sensing~\cite{5462827}, to the best of our knowledge the sparsity pattern in the form of \eqref{eq: matrix M} is unexplored.
Formally, the following is the main result of this paper: 
\begin{theorem}
\label{th: main}
There is an algorithm that solves problem~\eqref{pr: main} in polynomial time,
provided that 
$M$ is of the form \eqref{eq: matrix M} and
each of the blocks $A^1,\ldots,A^h$ and $C$ have a constant number of columns.
\end{theorem}
In particular, our algorithm solves $O(2^k h^{(k+1)^2(2 \bar \theta +3)})$ linear least squares problems, where $k$ is the number of columns of $C$, $n_i$ is the number of columns of block $A^i$, $\theta := \max\{n_i : i=1,\dots,h\}$, and $\bar \theta := (\theta - 1) \theta (\theta + 1)/2$.
We remark that the presence of $k$ and of $n_1,\dots,n_h$ in the exponent is expected, since the problem is NP-hard even if only one among $k,n_1,\dots,n_h$ is not a constant but a parameter which is part of the input.
This is because the block with a non-constant number of columns can be used to solve a general Subset Selection problem.


The generality of the techniques used in our proofs yields, in fact, a result that is stronger than Theorem~\ref{th: main}.
Indeed, the 2-norm in the objective plays no special role in our derivation, and our proofs can be tweaked to deal with a more general $p$-norm, for any fixed $p \ge 1$.



This paper is organized as follows.
After discussing related literature in Section~\ref{sec: literature}, in Section~\ref{sec: reduction} we present an overview of our algorithm and we introduce a reduction of problem~\eqref{pr: main} to separable form.
In Section~\ref{sec: diag} we prove Theorem~\ref{th: main} under the additional assumption that each $A^i$ block is a just a $1 \times 1$-matrix, meaning that the matrix $A$ is diagonal.
Then, in Section~\ref{sec: block} we give a complete proof of Theorem~\ref{th: main}, and so we study the more general case where each $A^i$ block can be any matrix with a constant number of columns.
The polynomial-time algorithm presented in Section~\ref{sec: block} clearly implies the polynomial solvability of the class of problems discussed in Section \ref{sec: diag}.
We present our algorithms in this order, since the proof for the diagonal case features many key ideas that will also be used in the block diagonal case, but with significantly more technicalities.

\subsection{Relevance of the topic and related work}
\label{sec: literature}

In Discrete Optimization, few examples of sparsity patterns are known that allow us to solve the problem in a running time that grows moderately with the ambient dimension.
The earliest occurrence of sparsity patterns studied in Discrete Optimization is sixty years old. It occurs in the famous maximum $s-t$ flow over time problem of Ford and Fulkerson~\cite{FF1958}. See \cite{H1995} and~\cite{MS2009} for surveys  about flow over time models.

Problem~\eqref{pr: main} and variations of it have been studied by several communities with quite different foci on theoretical and more computational aspects. Because of the breadth of the topics, a detailed description of the state-of-the-art for its many faces does not appear to be feasible. We list below several results with links to the computational complexity viewpoint on the subject, with comparison to our result, and point the reader to references. 


\paragraph{Algorithms assuming properties of eigenvalues of $M^{\top}M$}
Gao and Li~\cite{GaoLi13} present a polynomial time algorithm when the $d-l$ largest eigenvalues of $M^\top M$ are identical, provided that $l$ is a fixed number.  
Clearly, our result is significantly different since it is possible in our setting that the top eigenvalue of $M^TM$ is unique. 
Indeed, if $M$ is a diagonal matrix with distinct entires, we have a unique maximum eigenvalue of $M^TM$. 
However, in this case one could argue that it is possible to scale variables and arrive at a resulting matrix where all the eigenvalues of $M^TM$ are equal. 
On the other hand, consider the matrix:
\begin{align*}
M = 
\left(
\begin{array}{c|c}
I_n & \frac{1}{\sqrt{n}}\textbf{1}
\end{array}
\right),
\end{align*}
where $I_n$ is the $n$-times-$n$ identity matrix and $\textbf{1}$ is the column vector with each component equal to $1$. Now each column has an equal two-norm. It is easily verified that the largest eigenvalue of $M^{\top}M$ is unique (in fact equal to $2$ with the next eigenvalue being equal to $1$), thus showing that our results are different from those in~\cite{GaoLi13}.

\paragraph{Algorithms assuming sparsity properties of the covariance matrix} 
The covariance matrix is computed as $(M - \mathbf{1}a)^{\top}(M - \mathbf{1}a)$ where $\mathbf{1}$ is a column vector of $1$s and $a$ is a row vector  with $a_i$ being the average of the $i^{th}$ column of $M$. Associated with the covariance matrix is the so-called covariance graph, whose nodes correspond to columns of $M$, i.e.  the variables,  and an edge between two nodes $i$ and $j$ reflects the fact that the corresponding variables have non-zero covariance, i.e. the entry $(i,j)$ of $(M - \mathbf{1}a)^{\top}(M - \mathbf{1}a)$ is non-zero.  

Das and Kempe~\cite{DasKem08} give exact algorithms when the covariance graph is  sparse: either a tree, or an independent set with a fixed number of variables not participating in the independent set. Also, there exists a fully polynomial time approximation scheme for solving Subset Selection when the covariance matrix has constant bandwidth \cite{DasKem08}.  These results have been obtained by exploiting matrix perturbation techniques.	It is easy to see that in our setting the covariance matrix can be very dense. For example, consider again the matrix:
\begin{align*}
M = 
\left(
\begin{array}{c|c}
I_n & \frac{1}{\sqrt{n}}\textbf{1}
\end{array}
\right).
\end{align*}
If we examine the MLE estimator of the covariance matrix (see~\cite{DasKem08}) of $M$ (where rows are samples and columns are variables), then it easy to verify that $\cov(i,j) = -\frac{1}{n^2} \neq 0$ for $i, j \leq n$ and $i \neq j$. 


\paragraph{Algorithms assuming restricted isometry property of $M$} A series of papers~\cite{Don06,CanRomTao06,candes2005decoding, candes2007dantzig,candes2006stable} study this problem in the signal processing community under the name ``sparse approximation" with applications to the sparse signal recovery from noisy observations.  These papers prove various versions (of deterministic or randomized `noisy' setting)  of the following meta-theorem: ``Suppose there exists a sparse (noisy) solution $\hat{x}$ to $Mx = b$, that is $M\hat{x} = b + \textup{`noise'}$ and $\hat{x}$ is $\sigma$-sparse. If the matrix $M$ satisfies the ``restricted isometry condition", then there is a simple convex program whose optimal solution exactly recovers the solution $\hat{x}$." These results serve as a theoretical foundation of why  lasso~\cite{tibshirani1996regression} often works well experimentally. 

To derive the convex program, examine a dual version of problem~\eqref{pr: main}: given a  scalar $\epsilon >0$, a target vector $b$  and a matrix $M$, the task is to solve 
\begin{equation}\label{dual0}	\min \{ \abs{\supp(x)}  
\ : \
x \in \R^d,\;   \norm{Mx -b} \leq \epsilon\}.
\end{equation}
 or a simplified version of it  where the support-objective is replaced  by an $l_1$-norm objective:
\begin{equation}\label{dual1}	
\min \{ \norm{x}_1  
\ : \
x \in \R^d,\;   \norm{Mx -b} \leq \epsilon\}.
\end{equation}
The latter problem can be attacked using convex optimization tools.


It is also easy to show that the $M$ matrices we can provide a polynomial-time algorithm for, do not always satisfy the required restricted isometry condition. A key parameter is the $S$-restricted isometry constant~\cite{candes2007dantzig} of $M$, denoted as $\delta_S$, and defined as a scalar satisfying:
\begin{align*}
(1 - \delta_S)\norm{x}^2 \leq \norm{M_T x}^2 \leq (1 + \delta_S) \norm{x}^2, \ \forall x \in \mathbb{R}^{|T|},
\end{align*}
where $T$ is a subset of the index set of columns of $M$, $|T| \leq S$ and $M_T$ is the matrix with columns corresponding to indices in $T$.  One restricted isometry condition from~\cite{candes2005decoding}: $\delta_S + \theta_{S,S} + \theta_{S,2S} < 1$, guarantees the exact recovery of optimal solution for $\sigma \leq S$ by solving a convex program,  (where $\theta_{S,S} , \theta_{S,2S}$ are other non-negative parameters). Consider the matrix of the form 
\begin{align*}
M = 
\left(
\begin{array}{c|c}
D
&
\begin{matrix}
G \\ 
\textbf{0}_{n -3, 3}
\end{matrix}
\end{array}
\right),
\end{align*}
where $D$ is a $n$-times-$n$ diagonal matrix, $\textbf{0}_{n -3, 3}$ is the all-zero matrix with $n -3$ rows and $3$ columns, and $G$ is the matrix:
\begin{align*}
G = 
\left(
\begin{array}{ccc}
\frac{1}{\sqrt{6}}& \frac{1}{\sqrt{6}} & \frac{-2}{\sqrt{6}} \\ 
\frac{1}{\sqrt{6}}& \frac{-2}{\sqrt{6}}& \frac{1}{\sqrt{6}}   \\ 
\frac{-2}{\sqrt{6}}& \frac{1}{\sqrt{6}}& \frac{1}{\sqrt{6}}
\end{array}
\right).
\end{align*}
Then it is easy to see that $\delta_{3} \geq 1$ (and thus $\delta_3 + \theta_{3, 3} + \theta_{3, 6} \geq 1$), since $Gv =0$ for $v$ being the 3-dimensional all-ones vector.  Thus, for such a $M$ matrix the exact recovery results do not hold for $\sigma \geq 3$. In general, since in our result the extra $k$ columns (other than the block diagonal matrix) that are allowed in the $M$ matrix  can be arbitrary, it is very easy to destroy any potential restricted isometry condition.


\paragraph{Other computational approaches} Due to the vast applicability of Problem~\eqref{pr: main},  many computational approaches are available to solve experimentally instances of \eqref{pr: main} or variations of it. Among them are greedy algorithms (e.g., forward- and backward-stepwise selection~\cite{miller2002subset,efroymson1960multiple,couvreur2000optimality,das2011submodular,elenberg2018restricted} with provable approximation guarantees, forward-stagewise regression~\cite{efron2004least,HasTibFri}), branch and bound method~\cite{beale1967discarding,hocking1967selection} (e.g., the leaps and bounds procedure~\cite{furnival1974regressions}), and convex optimization (e.g., ridge regression~\cite{hoerl1970ridge}, the lasso~\cite{tibshirani1996regression,oymak2013squared}).
See also \cite{HasTibFri,david2017high} for further references on this topic.

\subsection{Overview of the algorithm and reduction to separable form}
\label{sec: reduction}

We briefly explain our overall strategy to prove Theorem~\ref{th: main}.
We first show that it suffices to consider problems of the form 
\begin{align}
\label{pr: main reduced}
\begin{split}
\min \quad & \norm{Ax - \pare{b - \sum_{\ell = 1}^{k} c_\ell \lambda_\ell}}^2 \\
\suc \quad & x \in \R^n, \ \lambda \in \R^{k} \\
& |\supp(x)| \le \es,
\end{split}
\end{align}
where $A \in \R^{m \times n}$ is block diagonal with blocks $A^i \in \R^{m_i \times n_i}$, for $i=1,\dots,h$, 
and $b,c_1,\dots,c_k \in \R^m$, $\es \in \N$.
In particular, $n = \sum_{i=1}^h n_i$.
More specifically, to prove Theorem~\ref{th: main}, it suffices to give a polynomial-time algorithm for problem~\eqref{pr: main reduced} under the assumption that $k,n_1,\dots,n_h$ are fixed numbers.


Next, we design an algorithm to solve the problem obtained from \eqref{pr: main reduced} by fixing the vector $\lambda \in \R^k$, namely
\begin{align*}
\opt(\es)_{|\lambda} & := \min \left\{\norm{A x - \pare{b - \sum_{\ell = 1}^k c_\ell \lambda_\ell}}^2 \ : \  x \in \R^n, \ |\supp(x)| \le \es \right\}.
\end{align*}
This algorithm has the additional property that it only needs to evaluate a polynomial number of inequalities in order to identify an optimal support of the problem.

Our next task is to partition all possible vectors $\lambda \in \R^k$ in a number of regions such that in each region the algorithm yields the same optimal support.
Normally, such a partition can be comprised of exponentially many regions, even for a polynomial-time algorithm.
However, the special property of our algorithm for $\opt(\es)_{|\lambda}$ allows us to obtain a partition formed only by a polynomial number of regions.

We then construct the set $\X$ of supports that contains, for each region, the corresponding optimal support.
In particular, the set $\X$ has the property that there is at least one optimal solution of \eqref{pr: main reduced} whose support is contained in some $\chi \in \X$.
Finally, for each support $\chi \in \X$, we solve the linear least squares problem obtained from problem~\eqref{pr: main reduced} by dropping the support constraint and by setting to zero all variables with indices not in $\chi$.
All these linear least squares problems can be solved in polynomial time (see, e.g.,~\cite{ShaBenMLbook}), and the best solution obtained is optimal to \eqref{pr: main reduced}.

\smallskip

The approach that we present relies on two key results.
The first is an analysis of the proximity of optimal solutions with respect to two consecutive ``support-conditions'' $\abs{\supp(x)} \le s$ and $\abs{\supp(x)} \le s+1$. 
This result is fairly technical and is presented in Section~\ref{sec: block simple}, as it is only needed in the general block diagonal case in order to obtain the algorithm for $\opt(\es)_{|\lambda}$.
The second result is a tool from Discrete Geometry known as the \emph{hyperplane arrangement theorem}.
We now briefly describe this theorem, which is needed to upper bound the number of regions in the partition of the $\lambda$-space that we construct.
A set $H$ of $n$ hyperplanes in a $k$-dimensional Euclidean space determine a partition of the space called the \emph{arrangement of $H$,} or the \emph{cell complex induced by $H$.}
The hyperplane arrangement theorem states that this arrangement consists of $O(n^k)$ polyhedra and can be constructed with $O(n^k)$ operations.
For more details, we refer the reader to the paper \cite{EdeOroSei86}, and in particular to Theorem~3.3 therein.

\bigskip


In the remainder of this section we show the reduction to form \eqref{pr: main reduced}.
Problem~\eqref{pr: main reduced} has a key benefit over the original Subset Selection problem~\eqref{pr: main} with matrix $M$ of the form \eqref{eq: matrix M}:
the support constraint is applied only to the variables 
associated with the columns of $A$. 
A key consequence is that, once we fix the variables that are not subject to support constraints, then the objective function is decomposable in a sum of quadratic functions, and each one depends only on the components of $x$ corresponding to a single block of $A$.
This property will be crucial in the design of our algorithm.

\begin{lemma}
\label{lem: reduction}
To prove Theorem~\ref{th: main}, it suffices to give a polynomial-time algorithm for problem~\eqref{pr: main reduced}, provided that $k,n_1,\dots,n_h$ are fixed numbers.
\end{lemma}

\begin{proof}
Consider problem~\eqref{pr: main} where the matrix $M$ is of the form \eqref{eq: matrix M}.
Recall that $k$ is the number of columns of $C$, and that $n_i$ is the number of columns of block $A^i$, for $i=1,\dots,h.$ 
Let $n := \sum_{i=1}^h n_i$ be the number of columns of $A$.
The structure of the matrix $M$ implies that the objective function of problem~\eqref{pr: main} can be written in the form 
\begin{align*}
\norm{Ax - \pare{b - c \mu - \sum_{\ell=1}^k c_\ell x_{n+\ell}}},
\end{align*}
For each subset $\L$ of $\{1,\dots,k\}$, we consider the subproblem obtained from \eqref{pr: main} by setting $x_{n+\ell} = 0$ for all $\ell \in \{1,\dots,k\} \setminus \L$, and by restricting the cardinality constraint to the $n$-dimensional vector $(x_1,\dots,x_n)$.
Formally,
\begin{align}
\label{pr: reduction inproof}
\begin{split}
\min \quad & \norm{Ax - \pare{b - c \mu - \sum_{\ell \in \L} c_\ell x_{n+\ell}}} \\
\suc \quad & x_1,\dots,x_{n} \in \R, \ x_{n + \ell} \in \R, \forall \ell \in \L, \ \mu \in \R \\
& |\supp(x_1,\dots,x_{n})| \le \es - |\L|.
\end{split}
\end{align} 

We claim that to solve problem~\eqref{pr: main} we just need to solve the $2^k$ distinct subproblems of the form \eqref{pr: reduction inproof} and return the best among the $2^k$ optimal solutions found.
To see this, let $(\bar x,\bar \mu)$ be feasible for \eqref{pr: main} and consider the subproblem~\eqref{pr: reduction inproof} corresponding to the set $\L := \{ \ell \in \{1,\dots,k\} : \bar x_{n+\ell} \neq 0\}$.
Then the restriction of $(\bar x,\bar \mu)$ obtained by dropping the zero components $\bar x_{n+\ell}$, for $\ell \in \{1,\dots,k\} \setminus \L$, is feasible for this subproblem and has the same objective value.
Now let $(\hat x, \hat \mu)$ be feasible for a subproblem~\eqref{pr: reduction inproof} corresponding to a subset $\L$ of $\{1,\dots,k\}$.
Consider now the extension of $(\hat x, \hat \mu)$ obtained by adding components $\hat x_{n+\ell} = 0$ for all $\ell \in \{1,\dots,k\} \setminus \L$.
This vector is feasible for \eqref{pr: main} and has the same objective value.
This concludes the proof of the equivalence of \eqref{pr: main} with all the $2^k$ subproblems \eqref{pr: reduction inproof}.

Notice that, the assumption in Theorem \ref{th: main} that each of the blocks $A^1,\ldots,A^h$ and $C$ have a constant number of columns means that in each subproblem~\eqref{pr: reduction inproof}, $n_1,\dots,n_h$ and $|\L|$ are fixed numbers.

Consider now a subproblem~\eqref{pr: reduction inproof}.
Let $\L' := \L \cup \{0\}$, and $k' := |\L'| = |\L|+1$.
We introduce a $k'$-dimensional vector $\lambda$ as
\begin{align*}
\lambda_\ell :=
\begin{cases}
\mu & \text{if } \ell = 0, \\
x_{n + \ell} & \text{if } \ell \in \L,
\end{cases}
\end{align*}
and we define $k'$ vectors $c'_\ell \in \R^m$, for $\ell \in \L'$, as
\begin{align*}
c'_\ell :=
\begin{cases}
c & \text{if } \ell = 0, \\
c_\ell & \text{if } \ell \in \L.
\end{cases}
\end{align*}
Finally, we define $\es':= \es - |\L|$.
It follows that a subproblem~\eqref{pr: reduction inproof} can be written in the form \eqref{pr: main reduced}, where the parameters $k,\es,c_\ell$ in \eqref{pr: main reduced} correspond to the parameters  $k',\es',c'_\ell$ just defined.
\end{proof}

To avoid confusion, we remark that the parameters $k,\es,c^\ell$ in \eqref{pr: main reduced} do not need to match the corresponding parameters $k,\es,c^\ell$ in model \eqref{pr: main}.

\section{The diagonal case}
\label{sec: diag}

In this section we prove a weaker version of Theorem~\ref{th: main} obtained by further assuming that the block $A$ of the matrix $M$ given by \eqref{eq: matrix M} is diagonal, i.e., each $A^i$ block  is a just a $1 \times 1$-matrix.
Then $A \in \R^{n \times n}$, and we denote by $a_{11},\dots,a_{nn}$ the diagonal entries of $A$.

\begin{theorem}
\label{th: diag}
There is an algorithm that solves problem~\eqref{pr: main} in polynomial time,
provided that 
$M$ is of the form \eqref{eq: matrix M},
$A$ is diagonal,
and $k$ is a fixed number.
\end{theorem}
In particular, our algorithm solves $O(2^k n^{2(k+1)})$ linear least squares problems.
The presence of $k$ in the exponent is expected, since the problem is NP-hard if $k$ is not a constant but a parameter which is part of the input.

For brevity, in this section, every time we refer to problem~\eqref{pr: main}, problem~\eqref{pr: main reduced}, and problem $\opt(\es)_{|\lambda}$, we assume that the matrices $M$ and $A$ satisfy the assumptions of Theorem \ref{th: diag}.

Instead of considering directly problem~\eqref{pr: main reduced}, 
we first consider the simpler problem $\opt(\es)_{|\lambda}$.
For notational simplicity, we define the vector 
\begin{align*}
b' := b - \sum_{\ell = 1}^k c_\ell \lambda_\ell.
\end{align*}

\begin{lemma}
\label{lem: diag simple}
Problem $\opt(\es)_{|\lambda}$ can be solved in polynomial time.
In particular, an optimal solution is given by
\begin{align*}
x^*_j := 
\begin{cases}
b'_j / a_{jj} & \text{if } j \in K \\
0 & \text{if } j \notin K,
\end{cases}
\end{align*}
where $K$ is a subset of $H := \{i \in \{1,\dots,n\} : a_{ii} \neq 0\}$ of cardinality $\max \{\es, |H|\}$ with the property that for each $i \in K$ and each $j \in H \setminus K$, we have $|b'_i| \ge |b'_j|$.
\end{lemma}

\begin{proof}
The objective function of Problem $\opt(\es)_{|\lambda}$ is $\sum_{j = 1}^n (a_{jj} x_j - b'_j)^2$.
The separability of the objective function implies that the objective value of a vector $x$ will be at least $\sum_{j \notin H} (b'_j)^2 + \sum_{j \in H \setminus \supp(x)} (b'_j)^2$.
The support constraint and the definition of the index set $K$ then imply that the optimum value of $\opt(\es)_{|\lambda}$ is at least $\sum_{j \notin H} (b'_j)^2 + \sum_{j \in H \setminus K} (b'_j)^2$.
The latter is the objective value of the vector $x^*$, thus it is an optimal solution to $\opt(\es)_{|\lambda}$.
\end{proof}


We are now ready to prove Theorem~\ref{th: diag}.

\begin{proof}[Proof of Theorem~\ref{th: diag}]
Lemma~\ref{lem: reduction} implies that, in order to prove Theorem~\ref{th: diag}, it suffices to give a polynomial-time algorithm for problem~\eqref{pr: main reduced}, where $A$ is diagonal, provided that $k$ is a fixed number.
In the remainder of the proof we show how to solve problem~\eqref{pr: main reduced}.


\begin{claim}
\label{claim: diag 2}
We can construct in polynomial time 
an index set $T$ with $|T| = O(n^{2k})$,
polyhedra $Q^t \subseteq \R^k$, for $t \in T$, that cover $\R^k$, 
and sets $\chi^t \subseteq \{1,\dots,n\}$ of cardinality $\es$, for each $t \in T$,
with the following property:
For every $t \in T$, and for every $\lambda$ such that $\lambda \in Q^t$, the problem $\opt(\es)_{|\lambda}$ has an optimal solution with support contained in $\chi^t$.
\end{claim}

\begin{cpf}
%
Lemma~\ref{lem: diag simple} implies that in order to understand the optimal support of problem $\opt(\es)_{|\lambda}$ for a fixed vector $\lambda$, it is sufficient to compare all quantities $|b_i - \sum_{\ell =1}^k (c_\ell)_i \lambda_\ell|$, for $i \in \{1,\dots,n\}$ with $a_{ii} \neq 0$.
So let $i$ and $j$ be two distinct indices in $\{1,\dots,n\}$ with $a_{ii}, a_{jj} \neq 0$.
We wish to subdivide all points $\lambda \in \R^k$ based on which of the two quantities $|b_i - \sum_{\ell =1}^k (c_\ell)_i \lambda_\ell|$ and $|b_j - \sum_{\ell =1}^k (c_\ell)_j \lambda_\ell|$ is larger.
In order to do so, consider the inequality
\begin{align*}
\abs{b_i - \sum_{\ell =1}^k (c_\ell)_i \lambda_\ell} \ge \abs{b_j - \sum_{\ell =1}^k (c_\ell)_j \lambda_\ell}.
\end{align*}
It is simple to check that the set of points in $\R^k$ that satisfy the above inequality can be written as the union of polyhedra using linear inequalities corresponding to the four hyperplanes in $\R^k$ defined by equations
\begin{align*}
b_i - \sum_{\ell =1}^k (c_\ell)_i \lambda_\ell &= 0, 
& b_i - \sum_{\ell =1}^k (c_\ell)_i \lambda_\ell  &= b_j - \sum_{\ell =1}^k (c_\ell)_j \lambda_\ell, \\
b_j - \sum_{\ell =1}^k (c_\ell)_j \lambda_\ell &= 0, 
& b_i - \sum_{\ell =1}^k (c_\ell)_i \lambda_\ell  &= - \pare{b_j - \sum_{\ell =1}^k (c_\ell)_j \lambda_\ell}.
\end{align*}

By considering these four hyperplanes for all possible distinct pairs of indices in $\{1,\dots,n\}$, we obtain at most $4(n^2-n) = O(n^2)$ hyperplanes in $\R^k$.
These hyperplanes subdivide $\R^k$ into a number of full-dimensional polyhedra. 
By the hyperplane arrangement theorem, this subdivision consists of polyhedra $Q^t$, for $t \in T$, where $T$ is an index set with $|T| = O((n^2)^k) = O(n^{2k})$. 
Since $k$ is fixed, $|T|$ is polynomial in $n$ and the subdivision can be obtained in polynomial time.

We now fix one polyhedron $Q^t$, for some $t \in T$.
By checking, for each hyperplane that we have constructed above, in which of the two half-spaces lies $Q^t$, we obtain a total order on all the expressions $|b_i - \sum_{\ell =1}^k (c_\ell)_i \lambda_\ell|$, for $i \in \{1,\dots,n\}$ with $a_{ii} \neq 0$.
The obtained total order is global, in the sense that, for each fixed $\lambda$ with $\lambda \in Q^t$, it induces a consistent total order on the values obtained by fixing $\lambda$ in the expressions $|b_i - \sum_{\ell =1}^k (c_\ell)_i \lambda_\ell|$.
This total order induces an ordering $i^t_1,i^t_2,\dots,i^t_n$ of the indices $i \in \{1,\dots,n\}$  with $a_{ii} \neq 0$ such that, for every $\lambda \in Q^t$, we have
\begin{align*}
\abs{b_{i^t_1} - \sum_{\ell = 1}^k (c_\ell)_{i^t_1} \lambda_\ell} \ge \abs{b_{i^t_2} - \sum_{\ell = 1}^k (c_\ell)_{i^t_2} \lambda_\ell} \ge \dots \ge \abs{b_{i^t_n} - \sum_{\ell = 1}^k (c_\ell)_{i^t_n} \lambda_\ell}.
\end{align*}
By Lemma~\ref{lem: diag simple}, for each $\lambda$ such that $\lambda \in Q^t$, the  problem $\opt(\es)_{|\lambda}$ has an optimal solution with support contained in $\chi^t := \{i^t_1,i^t_2,\dots,i^t_\es\}$.
\end{cpf}

Let $\X$ be the set containing all index sets $\chi^t$ obtained in Claim~\ref{claim: diag 2}, namely
\begin{align*}
\X := \{ \chi^t \ :  \ t \in T\}.
\end{align*}

\begin{claim}
\label{claim: diag 3}
There exists an optimal solution $(x^*,\lambda^*)$ of problem~\eqref{pr: main reduced} such that  
\begin{align*}
\supp(x^*) \subseteq \chi \text{ for some } \chi \in \X.
\end{align*}
\end{claim}

\begin{cpf}
Let $(x^*,\lambda^*)$ be an optimal solution of problem~\eqref{pr: main reduced}. 
Then $x^*$ is an optimal solution of the restricted problem $\opt(\es)_{|\lambda^*}$.
Let $Q^t$, for $t \in T$, be a polyhedron such that $\lambda^* \in Q^t$, and let $\chi^t \in \X$ be the corresponding index set.
By Claim~\ref{claim: diag 2}, the problem $\opt(\es)_{|\lambda^*}$ has an optimal solution $\tilde x$ with support contained in $\chi^t$.
This implies that the solution $(\tilde x,\lambda^*)$ is also optimal for problem~\eqref{pr: main reduced}.
\end{cpf}

For each $\chi \in \X$, each problem~\eqref{pr: main reduced}, with the additional constraints $x_i = 0$, for all $i \notin \chi$, is a linear least squares problem, since the cardinality constraint can be dropped, and it can then be solved in polynomial time.
The best solution among the obtained ones is an optimal solution of \eqref{pr: main reduced}.
To find an optimal solution to~\eqref{pr: main reduced}, we have solved $O(n^{2k})$ linear least squares problems.

Since in the proof of Lemma~\ref{lem: reduction} we have $k' \le k+1$, to find an optimal solution to~\eqref{pr: main} we need to solve $O(2^k n^{2(k+1)})$ linear least squares problems.
This concludes the proof of Theorem~\ref{th: diag}. 
\end{proof}

\bigskip

Note that the algorithm for $\opt(\es)_{|\lambda}$ given in Lemma~\ref{lem: diag simple} plays a key role in the proof of Theorem~\ref{th: diag}.
In fact, it allowed us to subdivide the space of the $\lambda$ variables in a polynomial number of regions such that in each region the algorithm given in Lemma~\ref{lem: diag simple} yields the same optimal support.
This part of the proof will be much more involved for Theorem~\ref{th: main}.
In fact, in the proof of our main result, in order to derive a similar subdivision, we need to study the proximity of optimal solutions with respect to two consecutive ``support-conditions'' $\abs{\supp(x)} \le s$ and $\abs{\supp(x)} \le s+1$.


\section{The block diagonal case}
\label{sec: block}

In this section we prove 
Theorem~\ref{th: main}.
Every time we refer to problem~\eqref{pr: main}, problem~\eqref{pr: main reduced}, and problem $\opt(\es)_{|\lambda}$, we assume that the matrices $M$ and $A$ satisfy the assumptions of Theorem \ref{th: main}.
In particular, the matrix $M$ is as in \eqref{eq: matrix M},
and the matrix $A \in \R^{m \times n}$ is block diagonal with blocks $A^i \in \R^{m_i \times n_i}$, for $i=1,\dots,h$.

Our overall strategy is similar to the one we used to prove Theorem~\ref{th: diag} and is described in Section~\ref{sec: reduction}.
Our first task is the design of a polynomial-time algorithm to solve the simpler problem $\opt(\es)_{|\lambda}$.

Note that there are many possible polynomial-time algorithms that one can devise for problem~$\opt(\es)_{|\lambda}$, and a particularly elegant one can be obtained with a dynamic programming approach similar to the classic dynamic programming recursion for knapsack \cite{MarTotBook}.
For each $i=1,\dots,h$, let $x^i \in \R^{n_i}$, $b^i \in \R^{m_i}$, and $c_\ell^i \in \R^{m_i} \ \forall \ell=1,\dots,k$ such that 
\begin{align*}
x=
\begin{pmatrix}
x^1 \\
\vdots \\
x^h
\end{pmatrix},
\qquad
b=
\begin{pmatrix}
b^1 \\
\vdots \\
b^h
\end{pmatrix},
\qquad
c_\ell=
\begin{pmatrix}
c_\ell^1 \\
\vdots \\
c_\ell^h
\end{pmatrix}.
\end{align*}
Denote by $\opt(1,\dots,i;j)$ the subproblem of $\opt(\es)_{|\lambda}$ on blocks $A^1, \dots, A^i$, with $i \in \{1,\dots,h\}$ and support $j$, with $j \in \{0,\dots,\es\}$, namely,
\begin{align*}
\min \quad & \norm{
\begin{pmatrix}
A^1 \\
& \ddots \\
&& A^i
\end{pmatrix}
x -
\pare{
\begin{pmatrix}
b^1 \\
\vdots \\
b^i
\end{pmatrix}
-
\sum_{\ell = 1}^k
\begin{pmatrix}
c_\ell^1 \\
\vdots \\
c_\ell^i
\end{pmatrix}
\lambda_\ell
}
}^2 \\
\suc \quad & x \in \R^{n_1 + \cdots + n_i} \\
& |\supp(x)| \le j.
\end{align*}
By exploiting the separability of the objective function, it can be checked that the recursion
\begin{align*}
\opt(1,\dots,i;j) & = \min_{t=0,\dots,j}\{\opt(1,\dots,i-1;j-t) + \opt(i;t)\}, \qquad \forall j \in \{0,\dots,\es\},
\end{align*}
for $i = 1,\dots,h$, yields an optimal solution for problem $\opt(\es)_{|\lambda}$.

We view this dynamic programming recursion as a ``horizontal approach'' where we fix the support but then proceed a recursion over the blocks. 
We did not see how this algorithm can be used in the proof of Theorem~\ref{th: main} to cover the space of the $\lambda$ variables with a sub-exponential number of regions such that in each region the algorithm yields the same optimal support.
In the next section, we present a combinatorial algorithm that not only can solve problem~$\opt(\es)_{|\lambda}$, but that also allows us to obtain a polynomial covering of the space of the $\lambda$ variables.
This algorithm is based on a sort of ``vertical approach'' where we incorporate all blocks simultaneously and then develop a recursion to enlarge the support condition.

\subsection{A proximity theorem}
\label{sec: block simple}

For every $i \in \{1,\dots,h\}$, and every $j \in \{0,\dots,n_i\}$, let $\val(i;j)$ be a real number.
For $\es \in \{0, \dots, \sum_{i=1}^h n_i \}$, we define
\begin{align}
\label{eq: sepa}
\val(\es) := \min \left\{ \sum_{i=1}^h \val(i;j_i) \ : \ \sum_{i=1}^h j_i = \es, \ j_i \in \{0,\dots,n_i\} \right\}.
\end{align}
Note that problem~$\opt(\es)_{|\lambda}$ can be polynomially transformed to a problem $\val(\es)$.
This can be seen by exploiting the block diagonal structure of the matrix $A$, and by defining $\val(i;j_i)$ to be the optimal value of the problem restricted to block $i$ and support $j_i$. 
The details of this reduction are given at the beginning of the proof of Theorem~\ref{th: main}.

\begin{definition}
Let $s$ and $q$ be nonnegative integers.
Given a feasible solution $j^s = (j^s_1,\dots,j^s_h)$ for $\val(s)$, we say that a feasible solution $j^{s+1} = (j^{s+1}_1,\dots,j^{s+1}_h)$ for $\val(s+1)$ is \emph{$q$-close} to $j^s$ if
\begin{align*}
\sum_{i \in I^-} (j^s_i - j^{s+1}_i) & = q,
\end{align*}
where $I^- := \{i \in \{1,\dots,h\} : j^{s+1}_i < j^s_i \}$.
\end{definition}
Clearly, if $j^{s+1}$ is $q$-close to $j^s$, we also have that
\begin{align*}
\sum_{i \in I^+} (j^{s+1}_i - j^s_i) & = q + 1,
\end{align*}
where $I^+ := \{i \in\{1,\dots,h\} : j^{s+1}_i > j^s_i \}$. This yields to a $l_1$-proximity bound of $\|j^{s+1} - j^{s} \ \|_1 \leq 2q+1.$

A \emph{weak composition} of an integer $q$ into $p$ parts is a sequence of $p$ non-negative integers that sum up to $q$.
Two sequences that differ in the order of their terms define different weak compositions.
It is well-known that the number of weak compositions of a number $q$ into $p$ parts is $\binom{q+p-1}{p-1} = \binom{q+p-1}{q}$.
For more details on weak compositions see, for example, \cite{HeuManBook}.

Our next result establishes  that optimal solutions for $\val(s)$ and $\val(s+1)$ are close to each other. 

%
%

\begin{lemma} [Proximity of optimal solutions] 
\label{lem: increase}
Given an optimal solution $j^s$ for $\val(s)$, there exists an optimal solution $j^{s+1}$ for $\val(s+1)$ that is $q$-close to $j^s$, for some $q \in \{0,\dots, (\theta - 1) \theta (\theta + 1)/2\}$, where $\theta := \max\{n_i : i=1,\dots,h\}$.
\end{lemma}

\begin{proof}
Let $j^{s+1}$ be an optimal solution for $\val(s+1)$ such that   
\begin{equation}
\label{minimal}
\sum_{i =1}^h |j^{s+1}_i - j^s_i| \text{ is minimal.}
\end{equation}
Let $I^+ := \{i \in\{1,\dots,h\} : j^{s+1}_i > j^s_i \}$ and $I^- := \{i \in \{1,\dots,h\} : j^{s+1}_i < j^s_i \}$.
Note that 
\begin{align*}
\sum_{i \in I^+} (j^{s+1}_i - j^s_i) = \sum_{i \in I^-} (j^s_i - j^{s+1}_i) + 1.
\end{align*}

For $p=1,\dots,\theta$, let
\begin{align*}
x_p & := |\{i \in I^+ : j^{s+1}_i - j^s_i = p \}| \\
y_p & := |\{i \in I^- : j^s_i - j^{s+1}_i  = p \}|,
\end{align*}
This yields to the two equations 
\begin{align*}
\sum_{i \in I^+} (j^{s+1}_i - j^s_i)  = \sum_{p = 1}^\theta p x_p \quad \text{ and } \quad 
\sum_{i \in I^-} (j^s_i - j^{s+1}_i)  = \sum_{p = 1}^\theta p y_p.
\end{align*}

If $\sum_{p = 1}^\theta p y_p \le (\theta - 1) \theta (\theta + 1)/2$, then the statement is verified. 
Otherwise,  $\sum_{p = 1}^\theta p y_p > (\theta - 1) \theta (\theta + 1)/2$.
This implies that there exists $v \in \{1,\dots,\theta\}$ with $y_v \ge \theta$.
To see this, note that if $y_p \le \theta - 1$, for all $p=1,\dots,\theta$, then we obtain
\begin{align*}
\sum_{p = 1}^\theta p y_p \le (\theta - 1) \sum_{p = 1}^\theta p = (\theta - 1) \theta (\theta + 1)/2.
\end{align*}
From the fact that $\sum_{p = 1}^\theta p x_p = \sum_{p = 1}^\theta p y_p + 1 > (\theta - 1) \theta (\theta + 1)/2$, it also follows that there exists $u \in \{1,\dots,\theta\}$ with $x_u \ge \theta$.

In particular we have that 
\begin{align*}
x_u \ge \theta \ge v \text{ and } y_v \ge u.
\end{align*}
Thus there exists a subset $\tilde I^+$ of $I^+$ such that $|\tilde I^+| = v$, and $j^{s+1}_i - j^s_i = u$ for all $i \in \tilde I^+$.
Symmetrically, there exists a subset $\tilde I^-$ of $I^-$ such that $|\tilde I^-| = u$, and $j^s_i - j^{s+1}_i = v$ for all $i \in \tilde I^-$.

Let $\tilde j^{s+1}$ be obtained from $j^{s+1}$ as follows
\begin{align*}
\tilde j^{s+1}_i = 
\begin{cases}
j^{s+1}_i - u = j^{s}_i & \text{if } i \in \tilde I^+ \\
j^{s+1}_i + v = j^{s}_i & \text{if } i \in \tilde I^- \\
j^{s+1}_i & \text{otherwise}.
\end{cases}
\end{align*}
Since $\sum_{i=1}^h \tilde j^{s+1}_i = \sum_{i=1}^h j^{s+1}_i - uv + uv = s+1$, we have that $\tilde j^{s+1}$ is a feasible solution for $\val(s+1)$.
Moreover, we have that 
\begin{align*}
\sum_{i =1}^h |\tilde j^{s+1}_i - j^s_i| = \sum_{i =1}^h |j^{s+1}_i - j^s_i| - 2uv < \sum_{i =1}^h |j^{s+1}_i - j^s_i|.
\end{align*}
In the remainder of the proof, we show that $\tilde j^{s+1}$ is an optimal solution for $\val(s+1)$.
This will conclude the proof, since it contradicts the choice of $j^{s+1}$ in Eq. (\ref{minimal}). 

Let $\tilde j^s$ be obtained from $j^s$ as follows
\begin{align*}
\tilde j^s_i = 
\begin{cases}
j^s_i + u = j^{s+1}_i & \text{if } i \in \tilde I^+ \\
j^s_i - v = j^{s+1}_i & \text{if } i \in \tilde I^- \\
j^s_i & \text{otherwise}.
\end{cases}
\end{align*}
Since $\sum_{i=1}^h \tilde j^s_i = \sum_{i=1}^h j^s_i + uv - uv = s$, we have that $\tilde j^s$ is a feasible solution for $\val(s)$.
Since $j^s$ is an optimal solution for $\val(s)$ we obtain
\begin{align*}
\sum_{i=1}^h \val(i;\tilde j^s_i) 
-
\sum_{i=1}^h \val(i;j^s_i) 
 & =  \sum_{i \in \tilde I^+ \cup \tilde I^-} (\val(i;j^{s+1}_i) - \val(i;j^s_i)) \ge 0.
\end{align*}

Consider now the objective value of the feasible solution $\tilde j^{s+1}$ for $\val(s+1)$.
\begin{align*}
\sum_{i=1}^h \val(i;\tilde j^{s+1}_i) 
& = \sum_{i=1}^h \val(i;j^{s+1}_i) 
+ \sum_{i \in \tilde I^+ \cup \tilde I^-} (\val(i;j^s_i) - \val(i;j^{s+1}_i)) \\
& \le \sum_{i=1}^h \val(i;j^{s+1}_i).
\end{align*}
This shows that the solution $\tilde j^{s+1}$ is optimal for $\val(s+1)$.
\end{proof}

Assume now that we know an optimal solution $j^s$ for $\val(s)$ and we wish to obtain an optimal solution for $\val(s+1)$.
By Lemma~\ref{lem: increase}, we just need to consider the feasible solutions for $\val(s+1)$ that are $q$-close to $j^s$, for some $q \le (\theta - 1) \theta (\theta + 1)/2 =:\bar \theta$.
We denote the family of these solutions by $\aug(j^s)$, formally
\begin{align}
\label{eq: aug}
\begin{split}
\aug(j^s) := \bigg\{j^{s+1} \ : \ & \  j^{s+1}_i \in \{0,\dots,n_i\}, i \in \{1,\dots,h\}, \ \sum_{i=1}^h j^{s+1}_i = s+1, \\
& \ j^{s+1} \text{ is $q$-close to $j^s$ for some $q \le \bar \theta$} \bigg\}.
\end{split}
\end{align}
For each solution $j^{s+1} \in \aug(j^s)$,
the corresponding objective function value is obtained from $\val(s)$ by adding
the difference,
\begin{align}
\label{eq: d}
d(j^s,j^{s+1}) : = \sum_{i\in I^+ \cup I^-} (\val(i;j^{s+1}_i) - \val(i;j^s_i)).
\end{align}
We denote by $\D(j^s)$ the family of the values $d(j^s,j^{s+1})$, for each solution $j^{s+1} \in \aug(j^s)$,
\begin{align*}
\D(j^s) := \{d(j^s,j^{s+1}) : j^{s+1} \in \aug(j^s)\}.
\end{align*}
From our discussions it follows that, in order to select the optimal solution for $\val(s+1)$, we only need to know which value in $\D(j^s)$ is the smallest.

\bigskip

We next define the set $\D$ as the union of all sets $\D(j^s)$, for any feasible solution $j^s$ of problem $\val(s)$, for any $s \in \{0, \dots, \sum_{i=1}^h n_i\}$.
Formally, the set $\D$ is defined as
\begin{align}
\begin{split}
\label{eq: qt}
\D := \bigg\{d(j^s,j^{s+1}) \ : \ & \ \text{$j^s$ is a feasible solution for $\val(s)$, for some $s \in \{0, \dots, \sum_{i=1}^h n_i \}$,} \\
& \ j^{s+1} \in \aug(j^s)\bigg\}.
\end{split}
\end{align}
The next result implies that the set $\D$ contains a number of values that is polynomial in $h$, provided that $\theta$ is fixed.
This fact could be surprising on the first glance 
since the number of feasible solutions $j^s$ for $\val(s)$ is of exponential order $\theta^h$.

\begin{lemma}
\label{lem: possibilities total}
The set $\D$ contains $O((\bar \theta+h)^{\bar \theta+1}\theta^{2\bar \theta+1})$ values.
\end{lemma}

\begin{proof}
We count the number of all possible values $d(j^s,j^{s+1})$ in $\D$.
Fix $q \in \{0,\dots,\bar\theta\}$, and let us consider all the values $d(j^s,j^{s+1})$ corresponding to a feasible solution $j^s$ for $\val(s)$ and a feasible solution $j^{s+1}$ for $\val(s+1)$ such that $j^{s+1}$ is $q$-close to $j^s$.
First, we construct the possible positive differences $d_i := j^{s+1}_i - j^s_i$, for $i \in I^+$.
Since the total sum of differences $d_i$, for $i \in I^+$, equals $q+1$, and the number of weak compositions of $q+1$ into $h$ parts is $\binom{q+h}{q+1}$, we conclude that there are $O((q+h)^{q+1})$ ways of choosing the set $I^+$ and constructing differences $j^{s+1}_i - j^s_i$, for $i \in I^+$.
Similarly, there are $O((q+h)^{q})$ ways of choosing the set $I^-$ and constructing differences $j^s_i - j^{s+1}_i$, for $i \in I^-$.
For each $i \in I^+ \cup I^-$, there are at most $\theta$ possible indices $j^s_i$, yielding a total number of $\theta^{|I^+ \cup I^-|} \le \theta^{2q+1}$ possible set of indices $j^s_i$, for $i \in I^+ \cup I^-$.
In total, we obtain $O((q+h)^{q+1} \theta^{2q+1})$ possible values.

From the fact that $q \in \{0,\dots,\bar\theta\}$, it follows that the cardinality of the set $\D$ is bounded by  $O((\bar \theta+h)^{\bar \theta+1}\theta^{2\bar \theta+1})$.
\end{proof}

\begin{proposition}
\label{prop: alg}
Given a total order on all the values in $\D$, we can construct an optimal solution for $\val(\es)$, for any $\es \in \{0, \dots, \sum_{i=1}^h n_i\}$, in time $O(\es(\bar \theta+h)^{\bar \theta+1}\theta^{2\bar \theta+1})$.
\end{proposition}

\begin{proof}
An optimal solution for $\val(0)$ is $j^0 = (0,\dots,0)$.
Let $s \in \{0, \dots, \sum_{i=1}^h n_i\}$ and assume that we have an optimal solution $j^s$ for $\val(s)$.
We show how we can construct an optimal solution $j^{s+1}$ for $\val(s+1)$.
Consider all the values in $\D(j^s)$.
Since $\D(j^s) \subseteq \D$, we can inquire a total order of $\D(j^s)$ from the given total order of $\D$.
Thus $\D(j^s)$ has a minimum element.
Since a minimum element can be found in linear time with respect to the cardinality of the set, 
we obtain a running time
$O(s(\bar \theta+h)^{\bar \theta+1}\theta^{2\bar \theta+1})$ as a consequence of Lemma~\ref{lem: possibilities total}.
In view of Lemma~\ref{lem: increase}, the solution in $\aug(j^s)$ corresponding to the minimum element in $\D(j^s)$ is an optimal solution $j^{s+1}$ for $\val(s+1)$.
This argument applied in an inductive manner leads to an optimal solution for $\val(\es)$ for any $\es \in \{0, \dots, \sum_{i=1}^h n_i\}$.
\end{proof}


\subsection{Proof of main result}
\label{sec: proof block}


We are now ready to present the formal proof of our main theorem.

\begin{proof}[Proof of Theorem~\ref{th: main}]
From Lemma~\ref{lem: reduction}, in order to prove Theorem~\ref{th: main}, we only need to show that there is a polynomial-time algorithm that solves problem~\eqref{pr: main reduced}, where $A \in \R^{m \times n}$ is block diagonal with blocks $A^i \in \R^{m_i \times n_i}$, for $i=1,\dots,h$, provided that $k,n_1,\dots,n_h$ are fixed numbers.
In the remainder of the proof we describe our algorithm to solve problem~\eqref{pr: main reduced}.

First, we consider problem $\opt(\es)_{|\lambda}$ and we rewrite it by exploiting the separability of the objective function.
For each $i=1,\dots,h$, let $x^i \in \R^{n_i}$, $b^i \in \R^{m_i}$, and $c_\ell^i \in \R^{m_i} \ \forall \ell=1,\dots,k$ such that 
\begin{align*}
x=
\begin{pmatrix}
x^1 \\
\vdots \\
x^h
\end{pmatrix},
\qquad
b=
\begin{pmatrix}
b^1 \\
\vdots \\
b^h
\end{pmatrix},
\qquad
c_\ell=
\begin{pmatrix}
c_\ell^1 \\
\vdots \\
c_\ell^h
\end{pmatrix}.
\end{align*}
Define the subproblem of $\opt(\es)_{|\lambda}$ on block $A^i$, with $i \in \{1,\dots,h\}$, and support $j$, for $j \in \{0,\dots,n_i\}$.
Formally, 
\begin{align*}
\opt(i;j)_{|\lambda} & := \min \left\{\norm{A^i x^i - \pare{b^i - \sum_{\ell = 1}^k c_\ell^i \lambda_\ell}}^2 \ : \ x^i \in \R^{n_i}, \ |\supp(x^i)| \le j \right\}.
\end{align*}
We can finally rewrite $\opt(\es)_{|\lambda}$ in the form
\begin{align}
\label{eq: sepa2}
\opt(\es)_{|\lambda} = \min \left\{ \sum_{i=1}^h \opt(i;j_i)_{|\lambda} \ : \ \sum_{i=1}^h j_i = \es, \ j_i \in \{0,\dots,n_i\} \right\}.
\end{align}
Note that in this new form, the decision variables are the integers $j_i$, for $i = 1,\dots,h$, and the variables $x$ do not appear explicitly.
We observe that we have reduced ourselves to the same setting described in Section~\ref{sec: block simple}.
In fact, for each fixed $\lambda$, each $\opt(i;j_i)_{|\lambda}$ can be calculated by solving a fixed number of linear least squares problems.
Thus, problem $\opt(\es)_{|\lambda}$ is now a problem of the form $\val(\es)$, as defined in \eqref{eq: sepa} and can be solved efficiently as a consequence of Proposition~\ref{prop: alg}.
Hence, problem $\opt(\es)_{|\lambda}$ can be solved for each fixed $\lambda$.
However, in order to solve our original problem, we have to solve $\opt(\es)_{|\lambda}$ for every $\lambda \in \R^k$.
In order to do this, we now think of $\opt(\es)_{|\lambda}$ and $\opt(i;j)_{|\lambda}$ as functions that associate to each $\lambda \in \R^k$ a real number.

Next, we define a space $\S$ that is an extended version of the space $\R^k$ of variables $\lambda_\ell$, for $\ell = 1,\dots,k$.
The space $\S$ contains all the variables $\lambda_\ell$, for $\ell = 1,\dots,k$, and it also contains one variable for each product of two variables $\lambda_{\ell_1} \lambda_{\ell_2}$, with $\ell_1,\ell_2 \in \{1,\dots,k\}$.
The dimension of the space $\S$ is therefore $(k^2+k)/2 \le k^2$.
Note that, for each $\lambda \in \R^k$, there exists a unique corresponding point in $\S$, that we denote by $\ext(\lambda)$, obtained by computing all the products $\lambda_{\ell_1} \lambda_{\ell_2}$, for $\ell_1,\ell_2 \in \{1,\dots,k\}$.

\begin{claim}
\label{claim: block 2}
We can construct in polynomial time an index set $T$ with $|T| = O(h^{k^2})$,
polyhedra $P^t \subseteq \S$, for $t \in T$, that cover $\S$, 
and sets $\upsilon^t(i;j)\subseteq \{1,\dots,n_i\}$ of cardinality $j$, for each $i \in \{1,\dots,h\}$, $j \in \{0,\dots,n_i\}$, and $t \in T$,
with the following property:
For every $i \in \{1,\dots,h\}$, $j \in \{0,\dots,n_i\}$, $t \in T$, and for every $\lambda$ such that $\ext(\lambda) \in P^t$,
the problem $\opt(i;j)_{|\lambda}$ has an optimal solution with support contained in $\upsilon^t(i;j)$.
\end{claim}

\begin{cpf}
Let $i \in \{1,\dots,h\}$.
First, we show that for every index set $\upsilon \subseteq \{1,\dots,n_i\}$, the best solution for problem $\opt(i;j)_{|\lambda}$ with support $\upsilon$ has an objective value that is a quadratic function in $\lambda$.
To see this, let $L$ be the linear subspace of $\R^{m_i}$ defined by $L := \{A^i x^i : \supp(x^i) \subseteq \upsilon\}$.
Consider the affine linear function $p : \R^k \to \R^{m_i}$ defined by $p^i(\lambda) := b^i - \sum_{\ell =1}^k c_\ell^i \lambda_\ell$.
Then the best solution for problem $\opt(i;j)_{|\lambda}$ with support $\upsilon$ has objective value
\begin{align}
\label{eq: best for one support}
\dist^2\left(L, p^i(\lambda)\right) 
= \norm{\proj_L (p^i(\lambda)) - p^i(\lambda)}^2.
\end{align}
The projection $\proj_L (p^i(\lambda))$ can be written as a linear function in $\lambda$.
In order to see this, as a consequence of the Gram-Schmidt orthogonalization, we can assume that the columns of $A^i$ are pairwise orthogonal. The projection of $p^i(\lambda)$ onto $L$ is simply the sum of scalar products of $p^i(\lambda)$ with the columns of the matrix $A^i$, which is a linear function.
Therefore the expression on the right hand side of \eqref{eq: best for one support} is a quadratic function in $\lambda$.

Let $j \in \{0,\dots,n_i\}$, and let $\upsilon^1, \upsilon^2$ be two different index sets contained in $\{1,\dots,n_i\}$ of cardinality $j$.
We wish to obtain a hyperplane that subdivides all points $\ext(\lambda) \in \S$ based on which of the two supports $\upsilon^1$ and $\upsilon^2$ yields a better solution for the problem $\opt(i;j)_{|\lambda}$.
To this end consider the equation whose left-hand side and right-hand side are two expressions of the type \eqref{eq: best for one support} corresponding to the two index sets $\upsilon^1$ and $\upsilon^2$, namely
\begin{align*}
\norm{\proj_{L^1} (p^i(\lambda)) - p^i(\lambda)}^2 
= \norm{\proj_{L^2} (p^i(\lambda)) - p^i(\lambda)}^2,
\end{align*}
where $L^\beta := \{A^i x^i : \supp(x^i) \subseteq \upsilon^\beta\}$ for $\beta \in \{1,2\}$.
Our argument implies that this is a quadratic equation.
Thus, by linearizing all the quadratic terms, we obtain a hyperplane in the space $\S$.
As desired, this hyperplane subdivides all points $\ext(\lambda) \in \S$ based on which of the two supports yields a better solution for the problem $\opt(i;j)_{|\lambda}$.

By considering the hyperplanes of this form corresponding to all possible distinct pairs of index sets in $\{1,\dots,n_i\}$ of cardinality $j$, we obtain fewer than $\binom{n_i}{j}^2$ hyperplanes in $\S$.
By considering these hyperplanes for all possible $j \in \{0,\dots,n_i\}$, we obtain at most $2^{2n_i}$ hyperplanes in $\S$, which is a fixed number since $n_i$ is fixed.
Then, by considering all $i \in \{1,\dots,h\}$, we obtain at most $\sum_{i=1}^h 2^{2n_i} \le h \max\{2^{2n_i} : i=1,\dots,h\} = O(h)$ hyperplanes in $\S$. 
These hyperplanes subdivide $\S$ into a number of full-dimensional polyhedra.
By the hyperplane arrangement theorem \cite{EdeOroSei86}, this subdivision consists of at most $O(h^{\dim(\S)}) = O(h^{k^2})$ polyhedra $P^t$, for $t \in T$.
Since $k$ is fixed, $|T|$ is polynomial in $h$ and the subdivision can be obtained in polynomial time. 

Let $i \in \{1,\dots,h\}$, $j \in \{0,\dots,n_i\}$, and $t \in T$.
By checking, for each hyperplane corresponding to two index sets contained in $\{1,\dots,n_i\}$ of cardinality $j$, in which of the two half-spaces lies $P^t$, we obtain a total order on all the expressions of the type \eqref{eq: best for one support} corresponding to index sets contained in $\{1,\dots,n_i\}$ of cardinality $j$.
The obtained total order is global, 
i.e.,
for each fixed $\lambda$ with $\ext(\lambda) \in P^t$, it induces a consistent total order on the values obtained by fixing $\lambda$ in the expressions \eqref{eq: best for one support}.
Consider a minimum element of the obtained total order, and the corresponding index set $\upsilon^t(i;j)\subseteq \{1,\dots,n_i\}$ of cardinality $j$.
This index set has then the property that problem $\opt(i;j)_{|\lambda}$ has an optimal solution with support contained in $\upsilon^t(i;j)$ for all $\lambda$ such that $\ext(\lambda) \in P^t$.
\end{cpf}

As in Section~\ref{sec: block simple}, let $\theta := \max\{n_i : i=1,\dots,h\}$, and $\bar \theta := (\theta - 1) \theta (\theta + 1)/2$.

\begin{claim}
\label{claim: block 3}
Let $t \in T$.
We can construct in polynomial time an index set $U^t$ with $|U^t| = O(h^{2 k^2 (\bar \theta +1)})$,
polyhedra $Q^{t,u} \subseteq P^t$, for $u \in U^t$, that cover $P^t$, 
and index sets $\chi^{t,u} \subseteq \{1,\dots,n \}$ of cardinality $\es$, for each $u \in U^t$, 
with the following property:
For every $t \in T$, $u \in U^t$, and for every $\lambda$ such that $\ext(\lambda) \in Q^{t,u}$,
the problem $\opt(\es)_{|\lambda}$ has an optimal solution with support contained in $\chi^{t,u}$.
\end{claim}

\begin{cpf}
To prove this claim, we first construct all polyhedra $Q^{t,u}$, for $u \in U^t$. 
Then we show how to construct the index sets with the desired property.

Let $s \in \{0,\dots,n\}$.
Let $j^s = (j^s_1,\dots,j^s_h)$ such that $j^s_i \in \{0,\dots,n_i\}$ for every $i \in \{1,\dots,h\}$, and $\sum_{i=1}^h j^s_i = s$.
Define $\aug(j^s)$ as in \eqref{eq: aug} and 
let $j^{s+1} \in \aug(j^s)$.
Define the sets $I^- := \{i \in \{1,\dots,h\} : j^{s+1}_i < j^s_i \}$ and $I^+ := \{i \in\{1,\dots,h\} : j^{s+1}_i > j^s_i \}$.
For $\lambda$ such that $\ext(\lambda) \in P^t$, we define the expression
\begin{align*}
d(j^s,j^{s+1})_{|\lambda} : = \sum_{i\in I^+ \cup I^-} (\opt(i;j^{s+1}_i)_{|\lambda} - \opt(i;j^s_i)_{|\lambda}).
\end{align*}
For $\lambda$ such that $\ext(\lambda) \in P^t$, consider the set $\D_{|\lambda}$ defined by
\begin{align*}
{\textstyle \D_{|\lambda}} := \bigg\{d(j^s,j^{s+1})_{|\lambda} \ : \ & \ \text{there exists $s \in \{0, \dots, \sum_{i=1}^h n_i\}$ such that } \\
& \ j^s \text{ is a feasible solution for } \opt(s)_{|\lambda} \text{ and } j^{s+1} \in \aug(j^s)\bigg\}.
\end{align*}
Note that for each fixed $\lambda$, each $d(j^s,j^{s+1})_{|\lambda}$ is a value of the form $d(j^s,j^{s+1})$, as defined in \eqref{eq: d}, and the set $\D_{|\lambda}$ reduces to the set $\D$, as defined in \eqref{eq: qt}.

Let $d(j^s,j^{s+1})_{|\lambda}$ and $d(k^s,k^{s+1})_{|\lambda}$ be two distinct expressions in $\D_{|\lambda}$.
We wish to subdivide all points $\ext(\lambda) \in \S$  based on which of the two expressions is larger.
In order to do so, consider the equation
\begin{align}
\label{eq: quadratic?}
d(j^s,j^{s+1})_{|\lambda} = d(k^s,k^{s+1})_{|\lambda}.
\end{align}
We show that \eqref{eq: quadratic?} is a quadratic equation in $\lambda$.
Consider a single $\opt(i;j)_{|\lambda}$ that appears in the expression defining $d(j^s,j^{s+1})_{|\lambda}$, and let $\upsilon^t(i;j)$ be the corresponding index set from Claim \ref{claim: block 2}.
Let $L$ be the linear subspace of $\R^{m_i}$ defined by $L := \{A^i x^i : \supp(x^i) \subseteq \upsilon^t(i;j)\}$, and let $p^i(\lambda) := b^i - \sum_{\ell =1}^k c_\ell^i \lambda_\ell$.
From Claim~\ref{claim: block 2}, 
for all $\lambda$ such that $\ext(\lambda) \in P^t$, 
we have that $\opt(i;j)_{|\lambda}$ can be written as the quadratic function in $\lambda$ of the form \eqref{eq: best for one support}, namely
\begin{align*}
\opt(i;j)_{|\lambda} = \norm{\proj_L (p^i(\lambda)) - p^i(\lambda)}^2.
\end{align*}
The expression $d(j^s,j^{s+1})_{|\lambda}$
is a linear combination of  expressions $\opt(i;j)_{|\lambda}$. Hence, it can also be written as a quadratic function in $\lambda$.
The same argument shows that also the expression $d(k^s,k^{s+1})_{|\lambda}$ can be written as a quadratic function in $\lambda$.
Hence \eqref{eq: quadratic?} is a quadratic equation in $\lambda$.
By linearizing all the quadratic terms, we obtain a hyperplane in the space $\S$.

As a consequence of Lemma~\ref{lem: possibilities total}, the set $\D_{|\lambda}$ contains $O(h^{\bar \theta +1})$ expressions. Thus, by considering the corresponding hyperplane for all possible distinct pairs of expressions in $\D_{|\lambda}$, we obtain a total number of $O(h^{2(\bar \theta +1)})$ hyperplanes.
These hyperplanes subdivide $\S$ into a number of full-dimensional polyhedra.
The hyperplane arrangement theorem \cite{EdeOroSei86} implies that this subdivision consists of at most $O((h^{2(\bar \theta +1)})^{\dim(\S)}) 
= O(h^{2 k^2 (\bar \theta +1)})$ polyhedra $R^u$, for $u \in U^t$.
Since $k$ and $\bar \theta$ are fixed, $|U^t|$ is polynomial in $h$ and the subdivision can be obtained in polynomial time.
Define $Q^{t,u} := P^t \cap R^u$, for every $u \in U^t$.

We now fix one polyhedron $Q^{t,u}$, for some $u \in U^t$.
By checking, for each hyperplane that we have constructed above, in which of the two half-spaces lies $Q^{t,u}$, we obtain a total order on all the expressions in $\D_{|\lambda}$.
The obtained total order
induces a consistent total order on the values obtained by fixing $\lambda$ in the expressions in $\D_{|\lambda}$, for each fixed $\lambda$ with $\ext(\lambda) \in Q^{t,u}$.
Since problem $\opt(\es)_{|\lambda}$ can be written in the form \eqref{eq: sepa2},  Proposition~\ref{prop: alg} implies that we can obtain an optimal support $\chi^{t,u} \subseteq \{1,\dots,\sum_{i=1}^h n_i \}$ for problem $\opt(\es)_{|\lambda}$, for each fixed $\lambda$ with $\ext(\lambda) \in Q^{t,u}$.
Note that, since the total order is independent on $\lambda$, also the obtained support is independent on $\lambda$.
Therefore the claim follows.
\end{cpf}

Let $\X$ be the set containing all index sets $\chi^{t,u}$ obtained in Claim~\ref{claim: block 3}, namely
\begin{align*}
\X := \{ \chi^{t,u} \ : \ t \in T, \ u \in U^t\}.
\end{align*}

\begin{claim}
\label{claim: block 4}
There exists an optimal solution $(x^*,\lambda^*)$ of problem~\eqref{pr: main reduced} such that 
\begin{align*}
\supp(x^*) \subseteq \chi \text{ for some } \chi \in \X.
\end{align*}
\end{claim}

\begin{cpf}
Let $(x^*,\lambda^*)$ be an optimal solution of problem~\eqref{pr: main reduced}. 
Then $x^*$ is an optimal solution of the restricted problem $\opt(\es)_{|\lambda^*}$.
Let $Q^{t,u}$, for $t \in T$, $u \in U^t$, be a polyhedron such that $\ext(\lambda^*) \in Q^{t,u}$, and let $\chi^{t,u} \in \X$ be the corresponding index set.
From Claim~\ref{claim: block 3}, the problem $\opt(\es)_{|\lambda^*}$ has an optimal solution $\tilde x$ with support contained in $\chi^{t,u}$.
This implies that the solution $(\tilde x,\lambda^*)$ is also optimal for problem~\eqref{pr: main reduced}.
\end{cpf}

For each $\chi \in \X$, each problem~\eqref{pr: main reduced}, with the additional constraints $x_i = 0$, for all $i \notin \chi$, is a linear least squares problem, since the cardinality constraint can be dropped, and it can then be solved in polynomial time.
The best solution among the obtained ones is an optimal solution of \eqref{pr: main reduced}.
To find an optimal solution to \eqref{pr: main reduced}, we have solved $O(h^{k^2}(h^{2(\bar \theta +1)})^{k^2}) = O(h^{k^2(2 \bar \theta +3)})$ linear least squares problems.

Since in the proof of Lemma~\ref{lem: reduction} we have $k' \le k+1$, to find an optimal solution to~\eqref{pr: main} we need to solve $O(2^k h^{(k+1)^2(2 \bar \theta +3)})$ linear least squares problems.
This concludes the proof of Theorem~\ref{th: main}.
\end{proof}

\ifthenelse {\boolean{SIOPT}}
{
\bibliographystyle{siamplain}
}
{
\bibliographystyle{plain}
}

\bibliography{biblio}

\end{document}